\documentclass[psamsfonts,12pt]{amsart}

\usepackage[margin=1in]{geometry}  % set the margins to 1in on all sides
\usepackage{graphicx}              % to include figures
\usepackage{amsmath}               % great math stuff
\usepackage{amsfonts}              % for blackboard bold, etc
\usepackage{amsthm}                % better theorem environments
\usepackage{verbatim}              % to comment out large swathes of text with the comment environment
\usepackage{indentfirst}           % to indent the first paragraph of a section
\usepackage{underscore}
\usepackage{enumitem}
\usepackage{mathtools} 
\usepackage{amssymb}
\usepackage[all]{xy}

\newtheorem{thm}{Theorem}[section]
\newtheorem{lem}[thm]{Lemma}
\newtheorem{prop}[thm]{Proposition}
\newtheorem{cor}[thm]{Corollary}
\newtheorem{conj}[thm]{Conjecture}

\theoremstyle{remark}
       \newtheorem*{rmk}{Remark}
\theoremstyle{remark}

\newcommand{\Mod}[1]{\ (\textup{mod}\ #1)}     
\newcommand{\QQ}{\mathbb{Q}}
\newcommand{\Qbar}{\overline{\QQ}}
\newcommand{\Pone}{\mathbb{P}^1} 
\newcommand{\Gal}{Gal}

 \newcommand{\Orb}{\operatorname{Orb}}
 \newcommand{\N}{\operatorname{N}} 
 \newcommand{\rank}{\operatorname{rank}} 
 \newcommand{\ord}{\operatorname{ord}}  
\usepackage{amssymb,fge}
\newcommand{\mysetminus}{\mathbin{\fgebackslash}}

\title{Galois groups of iterates of some unicritical polynomials}
\date{\today}
\author{Michael R. Bush} 
\address{Department of Mathematics, Washington \& Lee University; Lexington, VA 24450, USA}
\email{bushm@wlu.edu}
\author[Wade Hindes]{Wade Hindes}
\address{Department of Mathematics, The Graduate Center, City University of New York (CUNY); 365 Fifth Avenue, New York, NY 10016, USA}
\email{whindes@gc.cuny.edu}
\author[Nicole R. Looper]{Nicole R. Looper}
\address{Department of Mathematics, Northwestern University; 2033 Sheridan Road, Evanston, IL 60208, USA}
\email{nlooper@math.northwestern.edu}
\begin{document}
\maketitle
\begin{abstract} \normalsize We prove that the arboreal Galois representations attached to certain unicritical polynomials have finite index in an infinite wreath product of cyclic groups, and we prove surjectivity for some small degree examples, including a new family of quadratic polynomials. To do this, we use a combination of local techniques including the Chabauty-Coleman method and the Mordell-Weil sieve. 
\end{abstract}
\renewcommand{\thefootnote}{}

%%%%%%%%%%%%%%%%%%%%%%%%%%%%%%%%%%%%%%%%%%%%%%%%%%%%%%%%%%%%%%%%%%%%%%%%%%%%

\section{Introduction} 
Let $K$ be a number field. For a polynomial $\varphi(x)\in K[x]$, let $\varphi^n$ denote the $n$th iterate of $\varphi$, let $K_n(\varphi)$ be the splitting field of $\varphi^n$ over $K$ and let $G_{K,n}(\varphi):=Gal(K_n(\varphi)/K)$.  The groups $G_{K,n}(\varphi)$ form a projective system under the natural surjections $G_{K,n}(\varphi)\to G_{K,n-1}(\varphi)$, so that we may form the inverse limit 
\[G_{K}(\varphi):=\varprojlim_{n} G_{K,n}(\varphi).\]  
Much work has been done concerning the structure and size of $G_{K}(\varphi)$ in the case of quadratic polynomials~\cite{Wade, Jones1, Stoll}. For example, if $\varphi(x)\in\mathbb{Z}[x]$ is quadratic and critically infinite, then one expects that $G_{K}(\varphi)\leq [C_2]^{\infty}$ is a finite index subgroup~\cite[\S3]{Jones1}; here $C_d$ is a cyclic permutation group generated by a $d$-cycle and $[C_d]^\infty$ denotes the infinite iterated wreath product of $C_d$. Moreover, such a statement is known assuming the \emph{abc}-conjecture and an irreducibility condition~\cite[Prop. 6.1]{Khoa}. However, unconditional results are scarce~\cite{Jones-Manes,Stoll}, and up to this point, there are no examples in higher degree. In this article, we generalize a technique of Jones~\cite[Theorem 1.2]{Jones1} to produce polynomials $\varphi_p$ of prime degree $p\ge 3$ defined over $\QQ(\zeta_p)$ for which $G_{\mathbb{Q}(\zeta_p)}(\varphi_p)$ has finite index in $[C_p]^{\infty}$: 
\begin{thm}{\label{thm:unicrit}} Let $p$ be an odd prime, let $\zeta_p$ be a primitive $p$th root of unity, and let 
\[\varphi_p(x)=(x-1)^p+(2-\zeta_p).\]
Then there exists an explicit constant $C(p)$, depending only on $p$, such that    
\[\Big[[C_p]^\infty:G_{\mathbb{Q}(\zeta_p)}(\varphi_p)\Big]\leq C(p).\] 
Moreover, $G_{\mathbb{Q}(\zeta_p)}(\varphi_p)\cong [C_p]^\infty$ for $p=3,5, 7$.    
\end{thm}
In addition, we use the Chabauty-Coleman method~\cite{Poonen} in combination with the Mordell-Weil sieve~\cite{MW-Sieve} to produce the following family of quadratic polynomials with surjective Galois representations:   
\begin{thm}{\label{thm:quad}} Let $p\geq3$ be a prime and let 
\[\phi_p(x)=(x-p)^2+2p-p^2.\] 
Then $G_{\QQ}(\phi_p)\cong [C_2]^\infty$ in any of the following cases: 
\begin{enumerate}[topsep=5pt, partopsep=5pt, itemsep=2pt]
\item[\textup{(1)}] $p\equiv 2\Mod{3}$, 
\item[\textup{(2)}] $p\equiv 3\Mod{4}$, 
\item[\textup{(3)}] $p\equiv2\Mod{5}$,  
\item[\textup{(4)}] $p\equiv 3,6 \Mod{7}$,
\item[\textup{(5)}] $p\equiv 2,3,5,7,10 \Mod{11}$, 
\item[\textup{(6)}] $p\equiv 2,3,7,9,11 \Mod{13}$. 
\end{enumerate} 
Moreover, $G_{\mathbb{Q}}(\phi_p)\cong [C_2]^\infty$ for all primes $p< 5000$.    
\end{thm}
To prove Theorem~\ref{thm:quad}, we first show that $Gal_\QQ(\phi_p^3)\cong [C_2]^3$ for all primes $p\geq3$ in Lemma~\ref{lemma:smalln}. In particular, Theorem~\ref{thm:quad} and Lemma~\ref{lemma:smalln} provide evidence for the following conjecture:  
\begin{conj}Let $p\geq3$ be a prime and let \[\phi_p(x)=(x-p)^2+2p-p^2.\]
Then $G_{\QQ}(\phi_p)\cong [C_2]^\infty.$   
\end{conj} 
In addition to explicit techniques in the theory of rational points on curves, we make use of ideas developed in~\cite{Jones,Jones1}, as well as the computer algebra systems {\tt{Magma}}~\cite{Magma} and {\tt{Sage}}~\cite{Sage}. \\[5 pt]
\indent \textbf{Acknowledgements:} This research began at the May 2016 AIM workshop titled ``The Galois theory of orbits in arithmetic dynamics," and we thank AIM and the organizers of this workshop. The second author also thanks Michael Stoll for suggesting the use of the Mordell-Weil sieve to rule out residue classes when determining $\mathcal{C}_2(\mathbb{Q})$ below.

%%%%%%%%%%%%%%%%%%%%%%%%%%%%%%%%%%%%%%%%%%%%%%%%%%%%%%%%%%%%%%%%%%%%%%%%%%%%

\section{Main arguments}
In order to prove Theorem~\ref{thm:unicrit}, we make use of a slight modification of a lemma found in~\cite{Jones}. For $\phi(x)=(x-\gamma)^d+c\in K[x]$, each $K_n(\phi)$ is obtained from $K_{n-1}(\phi)$ by adjoining the $d$th roots of $\alpha_i-c$ for all roots $\alpha_i$ of $\phi^{n-1}(x)$. Writing $H_n:=Gal(K_n(\phi)/K_{n-1}(\phi))$, we then have an injection \[H_n \hookrightarrow (\mathbb{Z}/d\mathbb{Z})^{m}\] where $m$ is the degree of $\phi^{n-1}(x)$ and $n \geq 2$. This statement also holds when $n = 1$ provided the base field  $K$ contains a $d$th root of unity and we will make this assumption from this point forward. We say that $H_n$ is maximal when the injection is an isomorphism.

\begin{lem}{\label{maximality}} Let $d\ge 2$ be an integer, and let $K$ be a field of characteristic not dividing $d$. Let $\phi(x)=(x-\gamma)^d+c\in K[x]$. Suppose that $n\ge 2$ and that $\phi^{n-1}(x)$ is irreducible. Then $H_n$ is maximal if and only if $\phi^n(\gamma)$ is not a $p$th power in $K_{n-1}$ for any prime $p\mid d$.
\end{lem}
\begin{rmk} The proof proceeds exactly as in~\cite{Jones}, noting that adjoining the roots of $(x-\gamma)^d+c-\alpha_i$ for $\alpha_i$ a root of $\phi^{n-1}$ yields the same extension of $K(\alpha_i)$ as adjoining the roots of $x^d+c-\alpha_i$. \end{rmk}
Let $\Orb_\phi(P)$ denote the forward dynamical orbit of a point $P\in\Pone(\Qbar)$ under the action of $\phi$ and let $\hat{h}_\phi$ be the canonical height function associated to $\phi$; see \cite[Theorem 3.20]{Silv-Dyn}.
\begin{lem}{\label{lemma:0orbit}} Let $\phi(x)=(x-\gamma)^d+c \in \mathcal{O}_K[x]$. If $\frak{p}$ is a prime divisor of $\phi^n(\gamma)$, then $\frak{p}$ is a primitive prime divisor of $\phi^n(\gamma)$ if $\frak{p}$ does not divide any element of $\Orb_\phi(0)$.
\end{lem}
\begin{proof} If $\frak{p}\mid \phi^n(\gamma)$, then for any $1\le k<n$, we can write $\phi^n(\gamma)=\phi^{k}(\phi^{n-k}(\gamma))\equiv 0 \pmod{\frak{p}}$. Thus $\frak{p}$ is a primitive prime divisor of $\phi^n(\gamma)$ if and only if $\frak{p}\nmid \phi^k(0)$ for any $k<n$. 
\end{proof}
We now give a proof of Theorem~\ref{thm:unicrit}.
\begin{proof}[(Proof of Theorem~\ref{thm:unicrit})] 
Let $D_p=(-1)^{\frac{p-1}{2}}p^{p-2}$ be the discriminant of $K = \mathbb{Q}(\zeta_p)$, let $S_\infty$ be the archimedean places of $\mathbb{Q}(\zeta_p)$ and let $S$ be as follows:
\begin{equation} S=\big\{\text{primes}\;\mathfrak{p}\subseteq\mathbb{Z}[\zeta_p]\,:\, \N(\mathfrak{p})\leq (2/\pi)^{p-1} |D_p|^{1/2}\,\big\}\cup S_\infty;    
\end{equation}
here $\N(\mathfrak{p})=\#(\mathbb{Z}[\zeta_p]/\mathfrak{p})$ is the norm of the ideal $\mathfrak{p}$. It follows from~\cite[Theorem 5.4]{Lenstra} that the ring $\mathcal{O}_{K,S}$ of $S$-integers of $K$ is a principal ideal domain and that the free part of the unit group $\mathcal{O}_{K,S}^*$ is generated by elements $u_1, u_2, \dots u_t$ of height at most $(2/\pi)^{2p-2}|D_p|$. 

Considering the critical orbit $\Orb_{\varphi_p}(1)$ as a subset of $\mathcal{O}_{K,S}$, we can write
\begin{equation}{\label{decomp1}}
\varphi_p^n(1) = d_n\;y_n^p\,, \;\;\;\; \text{for some}\;\; d_n, y_n\in\mathcal{O}_{K,S}
\end{equation} 
with $0\leq v_\mathfrak{p}(d_n)\leq p-1$ for all $\mathfrak{p}\notin S$.  
%Note that by absorbing $p$-powers into $y_n$ we can assume that the exponents $0\leq n_i\leq p-1$. Similarly, since the prime ideals of $\mathcal{O}_{K,S}$ correspond to the prime ideals of $\mathcal{O}_K=\mathbb{Z}[\zeta_p]$ outside of $S$, we may assume that $0\leq v_\mathfrak{p}(d_n)\leq p-1$ for all $\mathfrak{p}\notin S$.  
We now use Lemma~\ref{lemma:0orbit} and our decomposition in~(\ref{decomp1}) to study primitive prime divisors in $\Orb_{\varphi_p}(1)$. To do this, note that $\varphi_p(0)=1-\zeta_p$ and $\varphi_p(1-\zeta_p)=1-\zeta_p$\,, from which it follows that the ideal generated by $(1-\zeta_p)$ is the only prime dividing the nontrivial elements of the orbit of $0$ (it is well known that $(1-\zeta_p)$ is the unique prime ideal above $p$). Moreover, $\varphi_p^n(1)\equiv 1\mod{(1-\zeta_p)}$ for all $n\geq0$, so that Lemma~\ref{lemma:0orbit} implies that $\varphi_p^n(1)$ and $\varphi_p^m(1)$ are coprime for all $n\ne m$. 

Now fix some $n\ge 1$ and consider $H_n(\varphi_p)=Gal(K_n(\varphi_p)/K_{n-1}(\varphi_p))$. Since any prime ramifying in $K_{n-1}$ must divide $\varphi_p^m(1)$ for some $m\le n-1$ by the discriminant formula in ~\cite[Lemma 2.6]{Jones1}, we see that Lemma~\ref{maximality} implies that $H_n$ is maximal unless $v_{\mathfrak{p}}(\varphi_p^n(1))\equiv 0 \Mod{p}$ for all primes $\mathfrak{p}$ of $K$; here we use that $\varphi_p^m$ is irreducible over $\QQ(\zeta_p)$ for all $m\geq1$, since $\varphi_p^m$ is Eisenstein at the prime $(1-\zeta_p)$. However, $v_\mathfrak{p}(\varphi_p^n(1))=v_{\mathfrak{p}}(d_n)+p\cdot v_{\mathfrak{p}}(y_n)$ for all $\mathfrak{p}\notin S$, so that if $H_n$ is not maximal, then $v_{\mathfrak{p}}(d_n)=0$ for all $\mathfrak{p}\notin S$ and so $d_n \in \mathcal{O}_{K,S}^*$. It follows that in this situation, we can rewrite~(\ref{decomp1}) as
\begin{equation}{\label{decomp2}}
\varphi_p^n(1)={\zeta_p}^{n_0}({u_1}^{n_1}\,{u_2}^{n_2}\dots \,{u_t}^{n_t})\;y_n^p\,,\;\; \text{for some} \;\; n_i \in \mathbb{Z}.        
\end{equation} 
We can further assume that $0 \leq n_i \leq p-1$ for all $i$ by absorbing $p$-powers into $y_n^p$.

The index bound in Theorem~\ref{thm:unicrit} now follows from an effective version of Siegel's integral point theorem applied to the superelliptic curve 
\[C_p^{(u)}: \;u\,Y^p=(X-1)^p+2-\zeta_p \]
and the $S$-integral point $(X,Y)=(\varphi_p^{n-1}(1), y_n)$; here $u$ is one of the finitely many $S$-units of the form $u={\zeta_p}^{n_0}({u_1}^{n_1}\,{u_2}^{n_2}\dots \,{u_t}^{n_t})$ for some $0\leq n_i\leq p-1$. To see this, let $s=\#S$, let $Q_S=\prod\N(\mathfrak{p})$ be the product of the norms of the finite primes of $S$, and let $h:\Qbar\rightarrow \mathbb{R}_{\geq0}$ be the standard logarithmic Weil height function on the algebraic numbers~\cite[\S3.1]{Silv-Dyn}. Then it follows from the height bound in~\cite[Theorem 2.1]{Siegel} that    
\begin{equation}{\label{htbd}}
h(\varphi_p^{n-1}(1))\leq (6ps)^{14p^6s}|D_p|^{2p^4}Q_S^{3p^4}e^{8p^5(p-1)\bar{h}(p,n)}; 
\end{equation} 
here $\bar{h}(p,n)$ is the height of the point $[1,u_n,a_{p-1},a_{p-2},\dots,a_1, 1-\zeta_p]$ in $\mathbb{P}^{p+1}(\Qbar)$ and the $a_i$ are the coefficients of $\varphi_p$: 
\[\varphi_p(x)=x^p+a_{p-1}x^{p-1}+\dots +a_1x+(1-\zeta_p).\] 
However, by construction there are at most $(p-1)(2/\pi)^{p-1}|D_p|^{1/2}$ primes in $S$: each prime $\mathfrak{p}\in S$ lies above a rational prime $q$ of size at most $(2/\pi)^{p-1}|D_p|^{1/2}$ and each rational prime $q$ lies below at most $p-1$ primes of $\mathbb{Q}(\zeta_p)$. Hence, $s=\#S$ and the the rank of the unit group $\mathcal{O}_{K,S}^*$ are bounded as follows: 
\begin{equation}{\label{rkbd}}
\;\;\;\;\;\;\;\,s\leq (p-1)+p\,|D_p|^{1/2}\;\;\;\;\text{and}\;\;\;\; \rank(\mathcal{O}_{K,S}^*)\leq \frac{p-1}{2}-1+p\,|D_p|^{1/2};
\end{equation} 
here we use that $2/\pi<1$. On the other hand, the height of a point $[x_0,x_1,\dots x_n]\in\mathbb{P}^{n}(\Qbar)$ is bounded above by $\sum h(x_i)$ so that 
\begin{equation}{\label{htbdcoeff}}
\bar{h}(p,n)\leq (p-1)\rank(\mathcal{O}_{K,S}^*)\log(D_p)+\big(p(p-1)+1\big)\log(2);
\end{equation} 
here we use our height bound on the generators $u_1,u_2, \dots u_t$ of the free part of $\mathcal{O}_{K,S}^*$ from~\cite[Theorem 5.4]{Lenstra} and the elementary height bounds: $h(x_1+x_2+\dots x_n)\leq \sum h(x_i)+\log(n)$ and $h(x_1x_2\dots x_n)\leq \sum h(x_i)$ for all $x_i\in\Qbar$. For these and other useful height estimates, see~\cite[\S3.4]{Siegel}.  Moreover, the $\log(2)$ above comes from the bound ${p\choose i}\leq 2^p$ for all $1\leq i\leq p-1$. 

Combining the estimates in~(\ref{htbd}), (\ref{rkbd}) and~(\ref{htbdcoeff}), we obtain the crude  bound:  
\begin{equation}{\label{crude}}
h(\varphi_p^{n-1}(1))\leq p^{16p^{p/2+9}+14p^{p/2+7}+84p^{p/2+6}+1.5p^{p/2+5}+2p^5-4p^4}.  
\end{equation} 
On the other hand, note that $\varphi_p^{n-1}(1)=\delta_p^{n-1}(0)+1$ for $\delta_p(x)=x^p+1-\zeta_p$. Moreover, $|h(x+1)-h(x)|\leq \log(2)$ for all $x\in\Qbar$: for a heavy-handed proof of this fact, once can apply the argument given in~\cite[Theorem 3.11]{Silv-Dyn} to the morphism $[x,y]\rightarrow [x+y,y]$ on $\mathbb{P}^1$. Furthermore,~\cite[Lemma 5.2]{Krieger} implies that 
\[|\hat{h}_{\delta_p}(x)-h(x)|\leq h(1-\zeta_p)+\log(2)\leq\log(4)\] 
for all $x\in\Qbar$; strictly speaking, this result is stated for polynomials $x^d+c$ for $c\in\mathbb{Q}$, although the rationality assumption is not necessary to establish this bound. Finally, by the standard transformation properties of the canonical height: $\hat{h}_{\delta_p}(\delta_p^m(x))=p^m \hat{h}_{\delta_p}(x)$ for all $m\geq1$ and $x\in\Qbar$; see, for instance~\cite[Theorem 3.20]{Silv-Dyn}. The bound in~(\ref{crude}) then reduces to 
\begin{equation}{\label{simplification}}
p^{n-1}\cdot\hat{h}_{\delta_p}(0)\leq p^{16p^{p/2+9}+14p^{p/2+7}+84p^{p/2+6}+1.5p^{p/2+5}+2p^5-4p^4}+ \log(8).
\end{equation} 
Therefore, it suffices to give a lower bound on $\hat{h}_{\delta_p}(0)$ to prove the finite index part of Theorem~\ref{thm:unicrit}. Such a bound is provided by the following general lemma, which is a simple consequence of~\cite[Exercises 3.3 and 3.17]{Silv-Dyn}:
\begin{lem}{\label{lem:lbd}}
Let $K/\QQ$ be a finite extension, let $\phi(x)\in K(x)$ be a rational map of degree $d$, and let $P\in\Pone(K)$ be a non-preperiodic point. Then
\[\hat{h}_\phi(P)\geq\frac{1}{d^{S_\phi}}\;\;\;\;\;\text{where}\;\;\;\; S_\phi:=12\cdot[K:\QQ]\cdot 2^{[K:\QQ]^2}\cdot(1+C_\phi)^{[K:\QQ]^2+[K:\QQ]}\,;\]
here $C_\phi$ is the constant bounding the difference $|\hat{h}_\phi(Q)-h(Q)|$ over all points $Q\in\Pone(\Qbar)$. 
\end{lem}
Hence,~(\ref{simplification}) and Lemma~\ref{lem:lbd} (applied to $\phi=\delta_p$) together imply that
\[n\leq 16p^{p/2+9}+14p^{p/2+7}+84p^{p/2+6}+1.5p^{p/2+5}+2p^5-4p^4+12(p-1)2^{(p-1)^2}(1+\log(4))^{p(p-1)}+2.\]
In particular, it follows that the index of $\Gal_{\QQ(\zeta_p)}(\varphi_p^m)\leq [C_p]^m$ is bounded independently of $m$ as claimed.   

Although it is nice to have an explicit upper bound on the iterates $n$ for which the groups $H_n$ are not maximal, these bounds are much too large to be useful in practice. For instance, when $p=3$ the bound above yields $n<20031664$. Therefore to prove surjectivity for $p=3,5,7$ we combine the techniques above with local computations. As a sketch, we compute a basis for the group $\mathbb{Z}[\zeta_p]^*/ (\mathbb{Z}[\zeta_p]^*)^p$ and rule out the possibility that $\varphi_p^n(1)=u_n\cdot y_n^p$ for all $1\leq n\leq 7$ by computing the absolute norm of $\varphi_p^n(1)$; here we use the fact that $\mathbb{Q}(\zeta_p)$ has class number one (hence it is not necessary to pass to a ring of $S$-integers) and that the norm of an algebraic unit is $\pm{1}$. In particular, it suffices to show that $\N_{\mathbb{Q}(\zeta_p)/\QQ}(\varphi_p^n(1))$ is not a $p$th power in $\mathbb{Z}$ for all $1\leq n\leq7$, to prove the maximality of Galois up to the $7$th stage; this can be easily verified with {\tt{Magma}}. To rule out larger $n\geq8$, we look at the critical orbit $\varphi_p^n(1)$ modulo small primes: $\mathfrak{q}=(2+\zeta_p)$, $(2-\zeta_p)$, $(3+\zeta_p)$, $(3-\zeta_p)$, $(2-3\zeta_p)$. The key here is that the sequence $\varphi_p^n(1) \Mod{\mathfrak{q}}$ is usually constant for all $n\geq8$, that is, the critical orbit fortuitously enters a fixed point. To make this argument explicit, we proceed in cases:\\
\\
\textbf{Case 1:} Let $p=3$, so that $\mathbb{Q}(\zeta_p)$ is an imaginary quadratic field with class number one and unit rank zero. Note that if $\varphi_3^n(1)$ has a decomposition such as that in~(\ref{decomp1}), then $\varphi_3^n(1)=\zeta_3^i\cdot y_n^3$ for some $0\leq i\leq2$ and some $y_n\in\mathbb{Z}[\zeta_3]$. On the other hand, if $\varphi_3^n(1)$ takes this form, we may assume that $n\geq8$, since $\N_{\mathbb{Q}(\zeta_3)/\QQ}(\varphi_3^n(1))$ is not a cube in $\mathbb{Z}$ for all $1\leq n\leq7$. However, if $\mathfrak{q}=(2-\zeta_3)$, then $\varphi_3(x)\equiv(x-1)^3\Mod{\mathfrak{q}}$ and hence $\varphi_3^n(1)\equiv 6\Mod{\mathfrak{q}}$ for all $n\geq2$; here we use that $\mathbb{Z}[\zeta_3]/\mathfrak{q}=\mathbb{F}_7$. However, the congruence $6\equiv 2,4\cdot y_n^3 \Mod{7}$ has no solutions, ruling out the possibility that $i=1,2$. On the other hand, if $\mathfrak{q}=(3+\zeta_3)$, then $\varphi_3(x)\equiv(x-1)^3+5\Mod{\mathfrak{q}}$ and $\varphi_3^n(1)\equiv 4\Mod{\mathfrak{q}}$ for all $n\geq3$; here again $\mathbb{Z}[\zeta_3]/\mathfrak{q}=\mathbb{F}_7$. However, $4$ is not a cube in $\mathbb{F}_7$, and we deduce that $i$ cannot be zero either.\\
\\ 
\textbf{Case 2:} Let $p=5$, so that $\mathbb{Q}(\zeta_5)$ is a degree $4$ extension with class number and unit rank equal to one. Moreover, one computes that $1+\zeta_5$ generates the free part of $\mathbb{Z}[\zeta_5]^*$. Hence, if $\varphi_5^n(1)$ has a decomposition such as that in~(\ref{decomp1}), then $\varphi_5^n(1)=\zeta_5^i\cdot (1+\zeta_5)^j\cdot y_n^5$ for some $0\leq i,j\leq4$ and some $y_n\in\mathbb{Z}[\zeta_5]$. On the other hand, if $\varphi_5^n(1)$ takes this form, then we may assume that $n\geq8$, since $\N_{\mathbb{Q}(\zeta_5)/\QQ}(\varphi_5^n(1))$ is not a $5$th power in $\mathbb{Z}$ for all $1\leq n\leq7$. However, if $\mathfrak{q}=(2-\zeta_5)$, then $\varphi_5(x)\equiv(x-1)^5\Mod{\mathfrak{q}}$ and hence $\varphi_5^n(1)\equiv 30\Mod{\mathfrak{q}}$ for all $n\geq2$; here we use that $\mathbb{Z}[\zeta_5]/\mathfrak{q}=\mathbb{F}_{31}$. However, one checks manually that $(i,j)\in\{(0,0),(1,1),(2,2),(3,3),(4,4)\}$ are the only exponents with solutions $30\equiv 2^i\cdot 3^j\cdot y_n^5\Mod{31}$. On the other hand, if $\mathfrak{q}=(2+\zeta_5)$, then $\varphi_5(x)\equiv(x-1)^5+4\Mod{\mathfrak{q}}$ and $\varphi_5^n(1)\equiv 5\Mod{\mathfrak{q}}$ for all $n\geq2$; here we use that $\mathbb{Z}[\zeta_5]/\mathfrak{q}=\mathbb{F}_{11}$. However, one checks that $(i,j)=(4,4)$ is the only remaining pair that has a solution $5\equiv (-2)^i\cdot (-1)^j\cdot y_n^5\Mod{11}$. Finally, if $\mathfrak{q}=(3+\zeta_5)$, then $\varphi_5(x)\equiv(x-1)^5+5\Mod{\mathfrak{q}}$ and $\varphi_5^n(1)\equiv 4\Mod{\mathfrak{q}}$ for all $n\geq3$; here we use that $\mathbb{Z}[\zeta_5]/\mathfrak{q}=\mathbb{F}_{61}$. Moreover,  $4\equiv (-3)^4\cdot (-2)^4\cdot y_n^5\Mod{61}$ has no solution, and we deduce that $(i,j)=(4,4)$ is also impossible. \\
\\ 
\textbf{Case 3:} Let $p=7$, so that $\mathbb{Q}(\zeta_7)$ is a degree $6$ extension with class number one and unit rank equal to two. Moreover, one computes with {\tt{Magma}} that $1+\zeta_7$ and $\zeta_7^4+\zeta_7$ generate the free part of $\mathbb{Z}[\zeta_7]^*$. Hence, if $\varphi_7^n(1)$ has a decomposition such as that in~(\ref{decomp1}), then $\varphi_7^n(1)=\zeta_7^i\cdot (1+\zeta_7)^j\cdot(\zeta_7^4+\zeta_7)^k\cdot y_n^7$ for some $0\leq i,j,k\leq6$ and some $y_n\in\mathbb{Z}[\zeta_7]$. On the other hand, if $\varphi_7^n(1)$ takes this form, then we may assume that $n\geq8$, since $\N_{\mathbb{Q}(\zeta_7)/\QQ}(\phi_7^n(1))$ is not a $7$th power in $\mathbb{Z}$ for all $1\leq n\leq7$.  However, if $\mathfrak{q}=(2-\zeta_7)$, then $\varphi_7(x)\equiv(x-1)^7\Mod{\mathfrak{q}}$ and hence $\varphi_7^n(1)\equiv -1\Mod{\mathfrak{q}}$ for all $n\geq2$; here we use that $\mathbb{Z}[\zeta_7]/\mathfrak{q}=\mathbb{F}_{127}$. However, setting $x\equiv i+k\Mod{7}$ and $y\equiv j+2k\Mod{7}$, one checks manually that $(x,y)=\{(0,0),(1,5),(2,3),(3,1),(4,6),(5,4),(6,2)\}$ are the only pairs of exponents with solutions $-1\equiv 2^i\cdot 3^j\cdot{18}^k\cdot y_n^7\Mod{127}$: here $2$, $3$ and $18$ are the images of the unit generators. In particular, there are only $49$ possible tuples $(i,j,k)$ that must be ruled out: each choice of $0\leq k\leq 6$ and $(x,y)$ in the collection above uniquely determines $i$ and $j$. We preculde these cases sequentially in $k$:  \\
\\
If $\boxed{k=0}$ and~(\ref{decomp1}) holds, then $(i,j)=\{(0,0),(1,5),(2,3),(3,1),(4,5),(5,4),(6,2)\}$ follows from the restrictions on $(x,y)$ above. Now let $\mathfrak{q}=(2+\zeta_7)$ so that $\mathbb{Z}[\zeta_7]/\mathfrak{q}=\mathbb{F}_{43}$ and $\varphi_7(x)\equiv(x-1)^7+4\Mod{\mathfrak{q}}$, and we compute that $\varphi_7^n(1)\equiv 3\Mod{\mathfrak{q}}$ for all $n\geq5$. One checks that among these restricted pairs,  $(i,j)=(6,2)$ is the only one having a solution to the congruence $3\equiv (-2)^i\cdot (-1)^j\cdot y_n^7\Mod{43}$. Finally, $(i,j)=(6,2)$ is ruled out modulo $\frak{q}=(3+\zeta)$: in this case $\mathbb{Z}[\zeta_7]/\mathfrak{q}=\mathbb{F}_{547}$ and $\varphi_7(x)\equiv(x-1)^7+5\Mod{\mathfrak{q}}$, and we compute that $\varphi_7^n(1)\equiv 407\Mod{\mathfrak{q}}$ for all $n\geq3$. Furthermore, the congruence $407\equiv (-2)^6\cdot (-1)^2\cdot y_n^7\Mod{547}$ has no solutions. \\
\\
If $\boxed{k=1}$ and~(\ref{decomp1}) holds, then $(i,j)=\{(6,5),(0,3),(3,4),(5,0),(1,1),(2,6),(4,2)\}$ follows from the restrictions on $(x,y)$ above. Again, let $\mathfrak{q}=(2+\zeta_7)$ so that $\mathbb{Z}[\zeta_7]/\mathfrak{q}=\mathbb{F}_{43}$ and $\varphi_7(x)\equiv(x-1)^7+4\Mod{\mathfrak{q}}$, and we compute that $\varphi_7^n(1)\equiv 3\Mod{\mathfrak{q}}$ for all $n\geq5$. One checks that among these restricted pairs, $(i,j)=(5,0)$ is the only one having a solution to the congruence $3\equiv (-2)^i\cdot (-1)^j\cdot y_n^7\Mod{43}$. Finally, as in the $k=0$ case, the pair $(i,j)=(5,0)$ is ruled out modulo $\frak{q}=(3+\zeta)$. \\
\\
If $\boxed{k=2,3,4,6}$ and~(\ref{decomp1}) holds, then one has seven possible pairs $(i,j)$ coming from the restrictions on $(x,y)$ above. For example, $(i,j)=\{(5,3),(6,1),(0,6),(1,4),(2,2),(3,0),(4,5)\}$ when $k=2$. As in the previous cases $k=0$ and $k=1$, only one pair remains after working modulo $\mathfrak{q}=(2+\zeta_7)$, and this exceptional case is ruled out modulo $\mathfrak{q}=(3+\zeta_7)$. \\
\\
If $\boxed{k=5}$ and~(\ref{decomp1}) holds, then $(i,j)=\{(2,4),(3,2),(4,0),(5,5),(6,3),(0,1),(1,6)\}$ follows from the restrictions on $(x,y)$ above. This case is slightly different. As usual, only the pair $(i,j)=(1,6)$ remains after working modulo $\mathfrak{q}=(2+\zeta_7)$. However, when $\mathfrak{q}=(3+\zeta_7)$, the congruence $\varphi_7^n(1)\equiv \zeta_7^1\cdot (1+\zeta_7)^6\cdot(\zeta_7^4+\zeta_7)^5\cdot y_n^7$ has solutions for all $n$ sufficiently large. Therefore, we need a new prime to finish this case. Let $\mathfrak{q}=(2+3\zeta_7)$ so that $\mathbb{Z}[\zeta_7]/\mathfrak{q}=\mathbb{F}_{463}$ and $\varphi_7(x)\equiv(x-1)^7+2-308\Mod{\mathfrak{q}}$, and we compute that $\varphi_7^n(1)\equiv 156\Mod{\mathfrak{q}}$ for all $n\geq5$. Moreover, the congruence $156\equiv \zeta^1\cdot {(1+\zeta_7)}^6\cdot {(\zeta_7^4+\zeta_7)}^5\cdot y_n^7\equiv-386\cdot y_n^7 \Mod{463}$ has no solutions.    

We have thus shown that the factorization~(\ref{decomp1}) is impossible for all $n\geq1$ when $p=3,5,7$. It follows that 
\[\Gal_{\mathbb{Q}(\zeta_3)}(\varphi_3^n)\cong [C_3]^n,\;\;\; \Gal_{\mathbb{Q}(\zeta_5)}(\varphi_5^n)\cong [C_5]^n \;\;\text{and}\;\; \Gal_{\mathbb{Q}(\zeta_7)}(\varphi_7^n)\cong [C_7]^n  \] 
for all $n\geq1$ as claimed.      
\end{proof}
The key fact that leads to our finite index result (and surjectivity in certain cases) is that the orbit of $0$ under $\varphi_p(x)=(x-1)^p+2-\zeta_p$ is strictly preperiodic. With this perspective, we produce a family of quadratic polynomials whose arboreal representations are surjective. In working with this family, we are greatly aided by explicit techniques in the theory of rational points on curves: specifically, we apply the Chabauty-Coleman method and the Mordell-Weil sieve.        

\begin{proof}[(Proof of Theorem~\ref{thm:quad})] It follows from~\cite[Proposition 4.6]{Jones1} that $\phi_p^n(x)$ is an irreducible polynomial over $\mathbb{Q}$ for all $p$ and all $n\geq1$. In fact, Jones shows the stronger statement that  $\phi_p^n(p)$ is not a square in $\mathbb{Q}$ for all $n\geq0$; see~\cite[Lemma 4.3]{Jones1}. In particular, for each $n\geq2$ it suffices to produce a prime $q_n$ satisfying: 
\begin{equation}{\label{primdiv}} 
v_{q_n}(\phi_p^n(p))\equiv 1\Mod{2}\;\;\;\;\;\text{and}\;\;\;\;\; v_{q_n}(\phi_p^i(p))=0\;\,\text{for all}\;\, 1\leq i\leq n-1
\end{equation} 
to prove that $\Gal_\mathbb{Q}(\phi_p^m)\cong [C_2]^m$ for all $m$; see~\cite[Theorem 3.3]{Jones1}. Note that $q_n$ will also depend on $p$, which we suppress in order to avoid cumbersome notation. To find such a $q_n$ we decompose $\phi_p^n(p)$ into a square and square-free part:        
\begin{equation}{\label{decomp}}
\phi_p^n(p)=\pm\, d_n\cdot y_n^2 \;\;\;\;\text{and}\;\;\; d_n=\prod_{i}q_{i}, 
\end{equation}
with the $q_{i}$ distinct prime numbers. Note that since $\phi_p^n(p)$ is not a square, $d_n$ must be nontrivial. Now, if no such prime $q_n$ as in~(\ref{primdiv}) exists, then for all $i$ there exists $n_i$ in the range $1\leq n_i\leq n-1$ such that $q_{i}\vert\phi_p^{n_i}(p)$. Hence, 
\begin{equation}{\label{congruence}}
0\equiv\phi_p^n(p)\equiv\phi_p^{n-n_i}(\phi_p^{n_i}(p))\equiv\phi_p^{n-n_i}(0)\Mod{q_i}. 
\end{equation} 
On the other hand $\phi_p(0)=\phi_p^2(0)=2p$, and it follows from~(\ref{congruence}) that $2p\equiv 0\Mod{q_i}$ for all $i$ since $n-n_i\neq0$. We deduce that $d_n=2^{\epsilon_1}\cdot{p}^{\epsilon_2}$ for some $\epsilon_i\in\{0,1\}$. However,    
\[ \phi_p(x)\equiv (x-1)^2+1\Mod{2}\;\;\;\text{and}\;\;\;\phi_p^n(p)\equiv1\Mod{2}\] 
for all $n\geq0$, and hence $\epsilon_1=0$. Likewise, it is easy to check that $\phi_p^2(p)>2p$ and that if $x>2p$ then $\phi_p^n(x)\geq 2p$ for all $n$. In particular, $\phi_p^n(p)>0$ for all $n\geq2$. Therefore,~(\ref{decomp}) reduces to    
\begin{equation}{\label{refinement}}
\boxed{
\phi_p^n(p)=p\cdot y_n^2 \;\;\;\;\;\text{for some}\;\;\; y_n\in\mathbb{Z},\, n\geq 2.} 
\end{equation}
Hence, it suffices to classify the primes $p$ for which~(\ref{refinement}) is impossible, to prove that the arboreal representations in Theorem~\ref{thm:quad} are surjective. To do this, we first classify the rational points on the curves 
\[\mathcal{C}_1:y^2=x^3-2x^2+2\;\;\;\text{and}\;\;\;\mathcal{C}_2: y^2=x^7-4x^6+4x^5+2x^4-4x^3+2,\] corresponding to the $\phi_p^2(p)=py^2$ and $\phi_p^3(p)=py^2$ cases, to rule out the possibility that~(\ref{refinement}) holds for $n=2,3$. We later show that for all primes $p<5000$, (\ref{refinement}) cannot hold when $n\ge 4$.

\begin{lem}{\label{lemma:smalln}} Let $p$ be an odd prime and let $\phi_p(x)=(x-p)^2+2p-p^2$. Then 
\[\Gal_\mathbb{Q}(\phi_p^3)\cong [C_2]^3\]
and $\phi_p(x)$ is stable over the rational numbers. 
\end{lem} 
\begin{rmk} The reader is encouraged to note that, unlike Theorem~\ref{thm:quad}, Lemma~\ref{lemma:smalln} assumes no congruence conditions on the prime. 
\end{rmk}
\begin{proof}[(Proof of Lemma~\ref{lemma:smalln})] Note that $\mathcal{C}_1$ is an elliptic curve in Weirerstrass form, hence all of the relevant arithmetic functions can be performed by {\tt{Magma}}. We compute that $\mathcal{C}_1(\mathbb{Q})\cong\mathbb{Z}$ with generator $(1,1)$ and that $(1,\pm{1})$ are the only integral points on $\mathcal{C}_1$ (points with integral $x$-coordinates). Therefore, there are no primes $p$ for which~(\ref{refinement}) holds when $n=2$.   

On the other hand, since $\mathcal{C}_2$ is a curve of genus $3$, the set $\mathcal{C}_2(\mathbb{Q})$ is finite and we prove that
\[\mathcal{C}_2(\mathbb{Q})=\{(1,\pm{1}),\infty\}.\]
To do this, let $\mathcal{J}_2$ be the Jacobian of $\mathcal{C}_2$. We compute with {\tt{Magma}} that $\#\mathcal{J}_2(\mathbb{F}_3)=24$ and $\#\mathcal{J}_2(\mathbb{F}_{11})=1351$. Moreover, since $\gcd\big(\#\mathcal{J}_2(\mathbb{F}_{3}),\#\mathcal{J}_2(\mathbb{F}_{11})\big)=1$ and $\mathcal{J}_2$ has good reduction modulo $3$ and $11$, we deduce that $\mathcal{J}_{2}(\mathbb{Q})$ has trivial torsion; see~\cite[Appendix]{Katz}. As for the free part of the Mordell-Weil group, a descent with {\tt{Magma}} shows that $\mathcal{J}_2(\mathbb{Q})$ has rank at most $2$. Conversely, the divisor class of $Q_0=[(1,1)-\infty]$ and the point on the Jacobian with Mumford representation $P_0=[x^2 - x - 1, -x + 1]$ are independent: they generate a non-cyclic subgroup of $\mathcal{J}_2(\mathbb{F}_3)\times\mathcal{J}_2(\mathbb{F}_5)$. Therefore, we have generators of a finite-index subgroup of $\mathcal{J}_2(\mathbb{Q})\cong\mathbb{Z}^2$, which is sufficient to try explicit forms of the Chabauty-Coleman method~\cite{Poonen,Siksek} in combination with the Mordell-Weil sieve~\cite{MW-Sieve} to determine $\mathcal{C}_2(\mathbb{Q})$. The first of these techniques applies since the genus of $\mathcal{C}_2$ is strictly larger than the rank of its Jacobian. 

Let $G$ be the subgroup of $\mathcal{J}_2(\QQ)$ generated by the divisors $P_0$ and $Q_0$ above. Since we cannot be sure that we capture the full Mordell-Weil group with $G$, we first show that the index $[\mathcal{J}_2(\QQ):G]$ is not divisible by the small primes in $S=\{2,3,5,7,11\}$. This is relatively easy: for each $\ell\in S$, we produce an auxiliary set of primes $S_\ell$ such that the induced map 
\[G/\ell G\rightarrow\prod_{\ell'\in S_\ell} \mathcal{J}_2(\mathbb{F}_{\ell '})/ \ell\mathcal{J}_2(\mathbb{F}_{\ell'})\]
is injective. It is straightforward to verify with {\tt{Magma}} that the sets $S_2=\{3,5\}$, $S_3=\{3,5\}$, $S_5=\{5,19\}$, $S_7=\{11,47\}$ and $S_{11}=\{13,37\}$ satisfy this property. In particular, if $\overline{G}_q$ and $\overline{\mathcal{J}_2(\QQ)}_q$ denote the images of $G$ and $\mathcal{J}_2(\QQ)$ in $\mathcal{J}_2(\mathbb{F}_q)$ respectively, then it follows from our exclusion of the small indices in $S$ that $\overline{G}_q=\overline{\mathcal{J}_2(\QQ)}_q$ for all $q\in S'=\{3,5,7,13\}$: the upshot of this step is that it allows us to be sure that any local information gained by reducing $\mathcal{J}_2(\mathbb{Q})$ modulo $q\in S'$ is captured instead by reducing $G$, which is concrete and explicitly known. We bracket this knowledge for now and proceed with the method of Chabauty and Coleman, which we briefly review; for a nice exposition, see~\cite{Poonen}. 

Let $\iota:\mathcal{C}_2(\QQ)\rightarrow\mathcal{J}_2(\QQ)$ be the Abel-Jacobi map given by $P\rightarrow [P-\infty]$. This map induces an inclusion of the rational points $\mathcal{C}_2(\QQ)\subset \mathcal{J}_2(\QQ) \subset\mathcal{J}_2(\QQ_q)$ into a $q$-adic Lie group, and since $\rank(\mathcal{J}_2(\QQ))=2$ is less than $\dim(\mathcal{J}_2(\QQ_q))=3$, there exists a non-zero regular $1$-form $\omega_q$ on $\mathcal{J}_2(\QQ_q)$ whose integral $P\rightarrow\int_{0}^P\omega_q$ annihilates $\mathcal{J}_2(\QQ)$; here for simplicity, we assume that $q$ is a prime of good reduction of $\mathcal{C}_2$. In particular, this $q$-adic integral kills the image of $\mathcal{C}_2(\QQ)$ in $\mathcal{J}_2(\QQ_q)$. On the other hand, on fibers of the reduction map $\pi_q:\mathcal{C}_2(\mathbb{Q}_q)\rightarrow\mathcal{C}_2(\mathbb{F}_q)$, called residue classes, this integral can be computed explicitly in terms of power series. Hence, one can use Newton polygons to bound $\#\mathcal{C}_2(\mathbb{Q})$. 

We carry out this procedure for $q=5$. Since $[\mathcal{J}_2(\QQ):G]$ is coprime to $5\cdot\#\mathcal{J}_2(\mathbb{F}_5)=900$, it follows that $P\rightarrow\int_{0}^P\omega_5$ kills $\mathcal{J}_2(\QQ)$ if and only if it kills $G$. Hence it suffices to compute $\omega_5$ using $G$. On the other hand, the embedding $\iota:\mathcal{C}_2\rightarrow \mathcal{J}_2$ induces an isomorphism between the regular $1$-forms $\Omega^1_{\mathcal{J}_2}(\mathbb{Q}_5)$ on $\mathcal{J}_2$ and the regular $1$-forms $\Omega^1_{\mathcal{C}_2}(\mathbb{Q}_5)$ on $\mathcal{C}_2$. Thus, via this identification, there exist $c_0$, $c_1$ and $c_2\in\mathbb{Z}_5$ such that $\omega_5=(c_2x^2+c_1x+c_0)/2y\, dx$. 

Let $\eta_i=\frac{x^i dx}{2y}$ for $0\leq i\leq2$ be the standard basis of $\Omega^1_{\mathcal{C}_2}$. We compute with the Coleman-integral function in {\tt{Sage}}~\cite{Jennifer,Sage} that 
\[\Big( \int_0^{Q_0}\eta_i\Big)_{0\leq i\leq2}=\big(3+ O(5^2),\, 3+3\cdot 5+O(5^2),\, 4+2\cdot 5+O(5^2)\big).\]                          
On the other hand, the divisor $P_0+18Q_0$ is in the kernel of reduction mod $5$, and we compute that $P_0+18Q_0=[U_1+U_2+U_3 -3\infty]$ for some points $U_j\in\mathcal{C}_2(\QQ_5)$. Again running the Coleman-integral function in {\tt{Sage}} we calculate that 
\[\bigg(\int_0^{P_0+18Q_0}{\eta_i}\bigg)_{0\leq i\leq2}=\bigg( \sum_{j=1}^3\int_0^{[U_j-\infty]}\eta_i\bigg)_{0\leq i\leq2}=\big(2\cdot 5+ O(5^2),\, 5+O(5^2),\, 3\cdot 5+O(5^2)\big).\]  
After scaling appropriately and reducing mod $5$, we deduce that $c_0\equiv0\Mod{5}$ and that $c_1\equiv 2c_2\Mod{5}$. Therefore, up to an irrelevant scaling factor, the differential $\omega_5$ that kills $\mathcal{J}_2(\QQ)$ reduces to 
\[\overline{\omega}_5=\frac{(x^2+2x)dx}{2y},\;\;\;\;\;\;\; \overline{\omega}_5\in\Omega_{\mathcal{C}_2}(\mathbb{F}_5).\]

Note that $\mathcal{C}_2(\mathbb{F}_5)=\{\infty,(1,\pm{1}), (3,\pm{2}), (4,\pm{2})\}$, so that if $\ord_{\overline{P}}\,(\overline{\omega}_5)>0$ for some point $\overline{P}\in\mathcal{C}_2(\mathbb{F}_5)$, then $\overline{P}=(3,\pm{2})$. Therefore, if $\overline{P}\neq(3,\pm{2})$ then the residue class of $\overline{P}$, i.e. the preimage of $\overline{P}$ via the reduction map $\mathcal{C}_2(\QQ_5)\rightarrow\mathcal{C}_2(\mathbb{F}_5)$, contains at most one rational point; see~\cite[Proposition 6.3]{Stoll-twists}. In particular, the residue classes of $\overline{P}=\infty$ and $\overline{P}=(1,\pm{1})$ contain exactly one rational point. Hence, it suffices to show that the residue classes of $\overline{P}=(3,\pm{2})$ and $\overline{P}=(4,\pm{2})$ contain no rational points, to prove that $\mathcal{C}_2(\mathbb{Q})=\{\infty, (1,\pm{1})\}$. To do this, we use the Mordell-Weil sieve~\cite{MW-Sieve}.  

In its simplest form, the Mordell-Weil sieve is a procedure for ruling out rational points in residue classes in the following way: let $S'$ be a set of primes of good reduction and consider the commutative diagram
\begin{displaymath}
    \xymatrix{ \ar[d]^{\pi_{S'}} \mathcal{C}_2(\QQ)\; \ar[r]^{\iota} & \;\mathcal{J}_2(\QQ) \ar[d]^{\alpha_{S'}} \\
               \displaystyle\prod_{q\in S'}\mathcal{C}_2(\mathbb{F}_q)\,\; \ar[r]^{\beta_{S'}} &\,\; \displaystyle\prod_{q\in S'}\mathcal{J}_2(\mathbb{F}_q)}
\end{displaymath}            
with the horizontal maps given by the basepoint at infinity and the vertical maps induced by reduction. Assuming we have generators of $\mathcal{J}_2(\QQ)$, we can compute the images of $\alpha_{S'}$ and $\beta_{S'}$ explicitly. Therefore, to rule out the existence of $P\in\mathcal{C}_2(\QQ)$ such that $\pi_{q_0}(P)=\overline{P}_{q_0}$ for some fixed $q_0\in S'$, we just need to check that
\[\beta_{S'}\Big(\{\overline{P}_{q_0}\}\;\,\times \displaystyle\prod_{q\in S'\mysetminus \{q_0\}}\mathcal{C}_2(\mathbb{F}_q)\Big)\;\bigcap\;\alpha_{S'}\Big(\mathcal{J}_2(\QQ)\Big)=\varnothing.\]
On the other hand, for each $q\in S'=\{3,5,7,13\}$, we have seen that any local information obtained from $\mathcal{J}_2(\mathbb{Q})$ can be obtained from $G$, i.e. that $\alpha_{S'}(G)=\alpha_{S'}(\mathcal{J}_2(\mathbb{Q}))$. Moreover, since $G$ is explicitly known to us, we can verify easily with {\tt{Magma}} that 
\[\beta_{S'}\Big(\{\overline{P}_{5}\}\;\,\times \displaystyle\prod_{q\in S'\mysetminus \{5\}}\mathcal{C}_2(\mathbb{F}_q)\Big)\;\bigcap\;\alpha_{S'}\big(G\big)=\varnothing,\;\;\;\;\text{for all}\;\, \overline{P}_{5}=(3,\pm{2}),\, (4,\pm{2}).\]
In particular, there exist no rational points $P\in\mathcal{C}_2(\QQ)$ reducing to $(3,\pm{2})$ or $(4,\pm{2})$ mod $5$. This completes the proof that $\mathcal{C}_2(\mathbb{Q})=\{\infty, (1,\pm{1})\}$ and the proof of Lemma~\ref{lemma:smalln}.     
\end{proof} 
To finish the proof of the Theorem 2, we use the local conditions above (and the fact that the critical orbit tends to end in a fixed point modulo small primes) to prove that~(\ref{refinement}) is impossible for $n\geq4$. We do this in cases: \\
\\
\textbf{Case 1:} If $p\equiv 2\Mod{3}$, then $\phi_p^n(p)\equiv 1\Mod{3}$ for all $n\geq2$. Therefore, if~(\ref{refinement}) holds for some $n\geq 4$, then $1\equiv\phi_p^n(p)\equiv2\cdot y_n^2\equiv 2\Mod{3}$ since $1$ is the only square in $\mathbb{F}_3^*$, and we reach a contradiction.    \\
\\
\textbf{Case 2:} Similarly, if $p\equiv 3\Mod{4}$, then $\phi_p^n(p)\equiv1\Mod{4}$ for all $n\geq1$. Hence, if~(\ref{refinement}) holds for some $n\geq 4$, then $1\equiv\phi_p^n(p)\equiv3\cdot y_n^2\equiv 3\Mod{4}$ since $1$ is the only non-zero square modulo $4$, and we reach a contradiction. \\
\\
\textbf{Case 3:} If $p\equiv 2\Mod{5}$, then $\phi_p^n(p)\equiv 4\Mod{5}$ for all $n\geq2$. Therefore, if~(\ref{refinement}) holds for some $n\geq 4$, then $4\equiv\phi_p^n(p)\equiv2\cdot y_n^2\equiv 2,3\Mod{5}$ since $1$ and $4$ are the only squares in $\mathbb{F}_5^*$. As in the previous cases, we reach a contradiction. \\ 
\\ 
\textbf{Case 4:} If $p\equiv 3,6\Mod{7}$, then $\phi_p^n(p)\equiv 1\Mod{7}$ for all $n\geq3$. Hence, if~(\ref{refinement}) holds for some $n\geq 4$, then we see that $1\equiv\phi_p^n(p)\equiv3\cdot y_n^2\equiv 3,5,6\Mod{7}$ since $\{1,2,4\}=(\mathbb{F}_7^*)^2$, yielding a contradiction.\\ 
\\
\textbf{Case 5:} If $p\equiv 2\Mod{11}$, then $\phi_p^n(p)\equiv 4\Mod{11}$ for all $n\geq2$ and $2$ is not a square in $\mathbb{F}_{11}$. Therefore~(\ref{refinement}) cannot hold for all $n\geq4$. Likewise, If $p\equiv 3\Mod{11}$, then $\phi_p^n(p)\equiv 6\Mod{11}$ for all $n\geq3$ and $6$ is not in the set $3\cdot (\mathbb{F}_{11}^*)^2$. Hence~(\ref{refinement}) cannot hold for any $n\geq4$. Similarly, if $p\equiv 5\Mod{11}$, then $\phi_p^n(p)\equiv 10\Mod{11}$ for all $n\geq3$ and $10$ is not in the set $5\cdot (\mathbb{F}_{11}^*)^2$. We deduce that~(\ref{refinement}) is impossible for all $n\geq4$. Finally, if $p\equiv 7,10\Mod{11}$, then $\phi_p^n(p)\equiv1\Mod{11}$ for all $n\geq3$ and neither $7$ nor $10$ is a square modulo $11$. It follows that that~(\ref{refinement}) cannot hold for all $n\geq4$.\\
\\
\textbf{Case 6:} If $p\equiv 2\Mod{13}$, then $\phi_p^n(p)\equiv 4\Mod{11}$ for all $n\geq2$ and $2$ is not a square modulo $13$. Therefore~(\ref{refinement}) is impossible. Likewise, if $p\equiv 3\Mod{13}$, then $\phi_p^n(p)\equiv 6\Mod{11}$ for all $n\geq4$ and $6$ is not in the set $3\cdot (\mathbb{F}_{13}^*)^2$. Hence~(\ref{refinement}) cannot hold for all $n\geq4$. On the other hand, if $p\equiv 9\Mod{13}$, then the orbit of $p$ enters a $2$-cycle: $\phi_p^n(p)\equiv 6,11\Mod{13}$ for all $n\geq3$. However, neither $6$ nor $11$ is a square modulo $13$, and we deduce that~(\ref{refinement}) is impossible. Finally, if $p\equiv 7,11\Mod{13}$, then $\phi_p^n(p)\equiv 1\Mod{13}$ for all $n\geq4$, and again~(\ref{refinement}) cannot hold for any $n\geq4$.  \\
\\  
\indent On the other hand, sieving through the $669$ primes $p<5000$, we see that only 
\[p=229, 1009, 1093, 1321, 1453, 3169, 3229, 3301, 3529, 4153, 4261, 4621, 4789  \]  
are not captured by any of the congruences above. Nonetheless, we can still show that~(\ref{refinement}) is impossible for all $n\geq4$ for these exceptional primes by working locally\,: for the primes $p=229, 1093, 1453, 3229, 3301, 4261, 4621, 4789$ work (mod $16$), for the primes $p=1009, 3529$ work (mod $19$), for $p=1321$ work (mod $17$), for $p=3169$ work (mod $53$), and finally for $p=4153$ work (mod $31$).             
\end{proof} 
\begin{rmk} Alternatively, it may be possible to use the explicit theory of heights on hyperelliptic genus $3$ Jacobians~\cite{Stoll-Heights} to prove that $G=\mathcal{J}_2(\QQ)$; this would shorten the proof of Lemma~\ref{lemma:smalln}.  
\end{rmk} 

It is likely that the techniques used to establish Theorem~\ref{thm:quad} can be adapted to other families of unicritical polynomials having zero as a strictly preperiodic point. For instance, we have the following example: 
\begin{prop}{\label{prop}} Let $p\geq3$ be an odd prime and let 
\[f_p(x)=(x-p)^2-p^2-1.\]
Then $\Gal_{\QQ}(f_p^3)\cong[C_2]^3$.   
\end{prop} 
\begin{proof} It follows from~\cite[Proposition 4.7]{Jones1} that $f_p^n$ is irreducible over $\QQ$ and that $f_p^n(p)$ is not a rational square in $\QQ$ for all $n$. Moreover, we compute that $\Orb_{f_p}(0)=\{-1,-2p\}$, so that~\cite[Theorem 3.3]{Jones1} implies that $\Gal_{\QQ}(f_p^m)\cong [C_2]^m$ unless there exists $2\leq n\leq m$ such that 
\[2^{\epsilon_1}\cdot p^{\epsilon_2}\cdot y_n^2=f_p^n(p);\]
here $\epsilon_i\in\{0,1\}$ and $\epsilon_1\cdot \epsilon_2\neq0$. Moreover, if $n$ is even, then $\epsilon_1=0$ for divisibility reasons. Similarly, if $n$ is odd, then $\epsilon_2=0$. In particular, we must rule out integral points $(p,y_n)$ on the curves 
\[X_1: y^2=x^3+2x^2+2x+2\;\;\;\text{and}\;\;\;\, X_2: 2y^2=x^8+4x^7+8x^6+10x^5+8x^4+4x^3-1\]
to prove the proposition. However, $X_1$ is an elliptic curve and {\tt{Magma}} computes that the only integral points on $X_1$ are $(1,\pm{1})$. Likewise, we compute with {\tt{Magma}} that the Jacobian $J(X_2)$ of $X_2$ has rank-zero, and that $\#JX_2(\mathbb{F}_3)=25$ and $\#JX_2(\mathbb{F}_5)=66$ are coprime. Hence, $J(X_2)(\QQ)$ is the trivial group ~\cite[Appendix]{Katz} and $X_2$ has no rational points.           
\end{proof} 
\begin{rmk} It follows from Proposition~\ref{prop} (and the analysis in its proof), that $G_{\QQ}(f_p)\cong [C_2]^\infty$ for all $p\equiv 2\Mod{5}$. This provides an example of how one might generalize Theorem~\ref{thm:quad}.   
\end{rmk}   

%%%%%%%%%%%%%%%%%%%%%%%%%%%%%%%%%%%%%%%%%%%%%%%%%%%%%%%%%%%%%%%%%%%%%%%%%%%%

\section{Appendix: Stability and Conjugation}
In this section, we make note of a technique for proving the irreducibility of certain polynomials obtained from Eisenstein polynomials via conjugation.
\begin{lem}{\label{lem:irre}} Let $K/\mathbb{Q}$ be finite. Let $p \in \mathbb{Z}$ be an odd prime and $\nu: K \rightarrow \mathbb{Z} \cup \{ \infty \}$ a normalized exponential valuation above $p$. Suppose $\nu(p) > 1$ and $f(x) = \sum_{i=0}^p c_i x^i \in K[x]$ satisfies the following conditions:
\begin{itemize}
\item[\textup{(i)}] $\nu(c_p) = 0$,
\item[\textup{(ii)}] $\nu(c_i) > 1$ for $1 \leq i \leq p-1$,
\item[\textup{(iii)}] $\nu(c_0) = 1$.
\end{itemize}
Then for all $\alpha \in K$ with $\nu(\alpha) \geq 0$, the polynomial  $f(x + \alpha) - c_p \alpha^p$ is Eisenstein with respect to~$\nu$.
\end{lem}
\begin{proof}
When $\alpha = 0$, the expression reduces to the polynomial $f$ which is clearly Eisenstein with respect to $\nu$. In fact, the given conditions on the coefficients are slightly stronger than needed. We now show that the stronger conditions imply the given statement for other choices of $\alpha \in K$ with $\nu(\alpha) \geq 0$,. 

If we write $\displaystyle f(x + \alpha)  = \sum_{j=0}^p b_j x^j$ then $\displaystyle b_j = \sum_{i=j}^p c_i \binom{i}{j} \alpha^{i - j}$ for $0 \leq j \leq p$. From this we see that $\nu(b_p) = \nu(c_p) = 0$, and for $0 \leq j \leq p - 1$ we have
\[    \nu(b_j)  \geq \min_{j \leq i \leq p} \nu \left(c_i \binom{i}{j}\right).
\]
For $1 \leq j \leq p - 1$, we have $\nu(c_p \binom{p}{j}) = \nu(p) > 1$. Combining this with the assumption that $\nu(c_i) > 1$ for $1 \leq i \leq p-1$, we see that $\nu(b_j) > 1$ for $1 \leq j \leq p - 1$. Finally, observe that $\displaystyle b_0 - c_p \alpha^p = c_0 + \sum_{i = 1}^{p-1} c_i \alpha^i$. Since $\nu(c_0) = 1$ and $\displaystyle \nu\left(  \sum_{i = 1}^{p-1} c_i \alpha^i \right) > 1$, it follows that $\nu \left( b_0 - c_p \alpha^p \right) = 1$. 

This shows that the coefficients of $f(x + \alpha) - c_p \alpha^p$ also satisfy the conditions in the statement of the lemma and hence this polynomial is Eisenstein with respect to $\nu$.
\end{proof}

\begin{cor}{\label{cor:eg}} Let $p$ be an odd prime, let $\zeta_p$ be a primitive $p$th root of unity, and let $i$ be an integer in the range $2\leq i\leq p$. Then all of the iterates of the polynomial 
\[\varphi_{(p,i)}(x) =  (x - \zeta_p^i)^p + (1+\zeta_p^i - \zeta_p)\] 
are irreducible over $\mathbb{Q}(\zeta_p)$. Moreover, $\varphi_{(p,i)}(0)=\varphi_{(p,i)}^2(0)=\zeta_p^i-\zeta_p,$ so that zero is strictly preperiodic for $\varphi_{(p,i)}$.  
\end{cor}
\begin{proof}
Let $\nu$ be the valuation on $\mathbb{Q}(\zeta_p)$ above $p$ so that $\nu(1 - \zeta_p) = 1$ and $\nu(p) = p - 1 > 1$. Now we apply Lemma~\ref{lem:irre} to $f(x) = x^p + (1 - \zeta_p)$ and $\alpha = -\zeta_p^i$, so that the polynomial $g(x)=(x-\zeta_p^i)^p+2-\zeta_p$ is Eisenstein at $\nu$. On the other hand, $g(x)-(1-\zeta_p^i)=\varphi_{(p,i)}(x)$; hence, it suffices to show that $\nu(\varphi_{(p,i)})=1$ to deduce that $\varphi_{(p,i)}(x)$ is Eisenstein at $\nu$. To do this, we compute that $\varphi_{(p,i)}(0)=\varphi_{(p,i)}^2(0)=\zeta_p^i-\zeta_p$ and that  \[\zeta_p^{p-i}\cdot\varphi_{(p,i)}(0)=1-\zeta_p^{p-i+1}.\] 
However, $p-i+1\not\equiv0\Mod{p}$, by the assumption $2\leq i\leq p$. Therefore, $(1-\zeta_p^{p-i+1})$ and $(1-\zeta_p)$ generate the same ideal in $\mathbb{Z}[\zeta_p]$, and we deduce that $\nu(\varphi_{(p,i)}^n(0))=1$ for all $n\geq1$. It follows that $\varphi_{(p,i)}^n$ is an Eisenstein polynomial with respect to $\nu$ for all $n\geq1$. 
\end{proof}
\begin{rmk} For $\varphi_{p}$ as in Theorem~\ref{thm:unicrit}, note that $\varphi_{p}=\varphi_{(p,p)}$ and that $\Orb_{\varphi_{(p,i)}}(\zeta_p^i)$ is finite. Therefore, it is likely that Theorem~\ref{thm:unicrit} holds if we replace $\varphi_p$ with $\varphi_{(p,i)}$ for any $2\leq i\leq p$. 
\end{rmk}

\end{document}